\documentclass[11pt,reqno]{amsart} 

\usepackage{amssymb,latexsym}
\usepackage[draft=true]{hyperref}
\usepackage{cite} 

\usepackage[height=190mm,width=130mm]{geometry} 

\theoremstyle{plain}
\newtheorem{theorem}{Theorem}

\newtheorem{corollary}{Corollary}

\theoremstyle{definition}
\newtheorem{definition}{Definition}

\theoremstyle{remark}

\numberwithin{equation}{section} 

\begin{document}

\title[ Stability  of   involutions on ...  ]{ Stability  of   involutions on Banach Algebra by  fixed point method   }%

\author{ N. Salehi }
\address{Department of Mathematics, Najafabad Branch, Islamic Azad University, Najafabad, Iran.  }%
\email{nsalehi@iaun.ac.ir}
\author {M. R. Velayati}%
\address{Department of Mathematics, Najafabad Branch, Islamic Azad University, Najafabad, Iran. }%
\email{ mvelayat@iaun.ac.ir}%

\thanks{}%
\subjclass[2010]{39B82 ;46L05;46L35 }%
\keywords{stability, Banach algebra, involution, $C^*$-algebra}%

\begin{abstract}
Using the fixed point method of the stability of the Jensen's functional equation, we proved the Hyers-Ulam-Rassias stability and super stability of involution on Banach algebra  and find some conditions that with them a Banach algebra with approximate involution is  $C^*$-algebra. 
\end{abstract}
\maketitle
\section{Introduction}
 The concept of stability for a functional equation arises when one replaces a functional equation by an inequality which acts as a perturbation of the equation. In 1940 S.M. Ulam \cite{u64} posed the first stability problem in the following year Hyers \cite{h41} gave a partial affirmative answer to the question  of Ulam.
 He has proved the following result.  Let's $E_1$ and $E_2$ are Banach space, for all $\epsilon>0$ there exists some $\delta>0$ such that for all  map  $f:E_1\to E_2$ that satisfies in
 $$\|f(x+y)-f(x)-f(y)\|\leq \delta,$$  for all $x,y\in E_1$, then there exits a unique additive map  $l:E_1\to E_2$ with 
 $$\|f(x)-l(x)\|\leq \epsilon$$
 
 for all $x,y\in E_1$.
 Hyers theorem was generalized by  Aoki\cite{a50} for additive mappings and by M. Rassias \citen{r78} for linear mapping by considering and unbounded Couchy difference.
In 1994, G\v{a}vruta\cite{g94} following the spirit of the approach of Th. M. Rassias proved the stability of the linear mapping  between Banach spaces for an arbitrary unbounded Cauchy difference.

In 1996, Isac and Rassias\cite{i96} were the first to provide applications of stability theory of functional equations for the proof of new fixed point theorems with applications.By fixed point methods of  several  functional equations have been extensively investigated by a number of authors( see  \cite{ch84,l2016}).

In 2003, L. C\u{a}dariu and V. Radu proved the Hyers-Ulam-Rassias stability of the Jensen's functional equation. 

 In this paper, we have used the technique of \cite{cadariu2003fixed} and obtained following results for stability of involution on Banach algebra by fixed point method. This is, if $f:A\to A$ is a approximately involution on Banach algebra $A$, then there exists an involution $I:A\to A$ which is near to $f$ more over, under some conditions on $f$, the algebra $A$ has  $C^*$-algebra structure with involution $I$.  

\begin{definition}
Let $X$ be a set. A function $d: X^2\to [0,\infty]$ is called a generalized metric on $X$ if and only if $d$
satisfies

(M1) $d(x, y)=0 $ if and only if $x= y$;

(M2) $d(x, y)=d(y, x)$, for all $x,y\in X$;

(M3) $d(x, z)\leq d(x, y)+ d(y, z)$ for all $x, y, z \in X$.
\end{definition}
We remark that the only difference between the generalized metric and the usual
metric is that the range of the former is permitted to include the infinity.
We now introduce one of the fundamental results of the fixed point theory.

\begin{theorem}[The alternative of fixed point \cite{diaz1968fixed,rus2008fixed}]\label{t0}
 Let (X,d) be a generalized complete metric space. Assume that
$T: X \to X$ is a strictly contractive operator with the Lipschitz constant $L < 1$. Then for each given $x\in X$, either $d(T^n(x),T^{n+1}(x))=\infty$ for all $n\geq 0$, or there exists a nonnegative integer $n_0$ such that

(i) $d(T^{n}(x),T^{n+1}(x)) < \infty$ for all
$n\geq n_0$;

(ii)  the sequence $\{T^n(x)\}$ converges to a fixed point $y^{\ast}$ of $T$;

(iii) $y^{\ast}$ is the unique fixed point of $T$ in $X^{\ast}=\{y\in X| d(T^{n_0}(x),y)<\infty\}$;

(iv) If $y \in X^{\ast}$, then
$$d(y, y^{\ast})\leq \frac 1{1-L} d(T(y),y).$$
\end{theorem}

\section{MAIN RESULTS}
Throughout this section, let $(E^k,\|\cdot\|_k): k\in\mathbb{N})$ be a multi-Banach algebra  and  $\mathbb{T}^1=\{z\in \mathbb{C}: \|z\|=1\}$ and for each $n_0\in \mathbb{ N}$ suppose $\mathbb{T}_{\frac{1}{n_0}}^1:=\{e^{i\theta}: 0\leq \theta\leq\frac{1}{n_0}\}$.  For convenience, we use the following abbreviation for a given mapping $f:E\to E$
$$D_{\lambda}f(x,y)=2 \overline{\lambda}f(\frac{x+y}{2})-fT( x)-fT( y),$$
for $x,y\in E$ and $\lambda \in \mathbb{C}$. 

Let us  recall some necessary definitions.

Let $A$ be an algebra over $\mathbb{C}$, then an involution on $A$ is a mapping $*:A\to A$ with 
$$*(a)=a^*$$
such that:

(i)$a^{**}=a$ for all $a\in A$.

(ii)$T( a+\mu b)^*=\overline{\lambda}a^*+\overline{\mu}b^*$ for all $a,b\in A$ and $\lambda,\mu\in \mathbb{C}$.

(iii)$(ab)^*=b^*a^*$ for all $a,b\in A$.\\
A $C^*$-algebras is a (non-zero) Banach-algebra with an involution such that 
$$\|a^*a\|=\|a\|^2.$$

We shall use the technics in \cite{maj2013} to prove the following two theorems.

\begin{theorem}\label{t1}
	suppose  $E$  be  Banach algebra  and $\phi:E^{2}\to [0,\infty)$ is a given function and there exists constant  $L$, $0 < L < 1$, such that
	\begin{equation}\label{e1}
	\phi (q_ix,q_iy)\leq q_i L\phi(x,y)
	\end{equation}
	for all $x \in E$ , where  $q_0=2$ and $q_1=\frac{1}{2}$ 
	if  $f:E\to E$  satisfies $f(0)=0$ , such that
	\begin{equation}\label{e2}
	\|D_{\lambda }f(x,y) \| \leq \phi(x,y)
	\end{equation}
	\begin{equation}\label{e3}
	\|f(xy)-f(y)f(x)\|\leq \phi(x,y)
	\end{equation}
		\begin{equation}\label{e4}
\lim_m q_i^{-m}f(q_i^m\lim_n q_i^{-n}f(q_i^{n}x))=x
	\end{equation}
	for all $x,y\in E$ and $\lambda \in \mathbb{T}_{\frac{1}{n_0}}^1$, then there is a unique involution mapping $I:E\to E$ 	which satisfies 
	\begin{equation}\label{e5}
	\|I(x)-f(x)\|\leq \frac{L^{1-i}}  {1- L} \phi (x,0)
	\end{equation}
	moreover if 
	\begin{equation}\label{e6}
	\Bigl|  \|xf(x)\|-\|x\|^2 \Bigr|\leq \phi (x,x)
		\end{equation}
	 for all  $x \in E$, then $E$ is a $C^*$-algebra with involution $x^*=I(x)$, for all $x\in E$. 
	
\end{theorem}
\begin{proof}
	If we define
	$$X=\{ g: E\to E | \ \ g(0)= 0\}$$
	and introduce a generalized metric on $X$ as follows:
	\begin{align*} d(g,h)= \inf\{ {c\in [0,\infty]} :  \| g(x)- h(x)\| \leq c\phi(x,0), \ \   {\rm for\ all\ } x \in E\}
	\end{align*}
	It is easy to show that (X, d) is complete.
	We define an operator $T: X \to X$ by
	$$T( g)(x)=\frac{ g(q_ix)}{q_i}$$
	for all $x \in E$. First, we assert that $T$ is strictly contractive on $X$. Given $g, h\in X$, let $c\in [0,\infty]$ be an arbitrary constant with $d(g, h)\leq c$, i.e.,
	$$\| g(x)- h(x)\| \leq c\phi(x,0),$$
	for all $ x\in E$. If we replace $x$ in the last inequality with $q_ix$ and make use of (\ref{e1}),
	then we have
	\begin{align*} \| T g(x)- T h(x)\|=q_i^{-1} \ \| g(q_ix)- h(q_ix)\|  \leq q_i^{-1}\  c\phi (q_ix,0)\leq Lc\phi(x,0)
	\end{align*}
	for every $ x_1,\ldots ,x_k\in E$, i.e., $d(T(g),T(h))\leq Lc$. Hence, we conclude that $d(T(g),T(h) )\leq L d(g,h)$  for any $g, h\in X$.

	Next, we assert that $d(T( f), f ) < \infty$. If $ i=0$ we replace $x$ with  $2t$ , $y$ with $y$ and $\lambda$  with $1$  in (\ref{e2})  then (\ref{e1}) establishes
	\begin{align*}
	\|2 f( \frac{2t}{2})-f(2t)\|\leq \phi(2t,0) 
	&\Rightarrow	\| f(t)-\frac{1}{2}f(2t)\|\leq \frac{1}{2}\phi(2t,0)  \leq L\phi(t,0)
	\\ &\Rightarrow \| T (f)( x)-f(x)\|\leq L\phi(x,0)
	\end{align*}
	for any $x \in E$, i.e.,
	\begin{equation}\label{e7}
	d(T (f),f)\leq L\leq\infty
	\end{equation}
	
	 If $ i=1$  we replace $y$ with $0$ and  $\lambda$ with $1$  in (\ref{e2})  then (\ref{e1}) establishes
	\begin{align*}
	\|2 f( \frac{x}{2})-f(x)\|\leq \phi(x,0) 
	\Rightarrow \|T( f( x))-f(x)\|\leq \phi(x,0)
	\end{align*}
	for any $x \in E$, i.e.,
	\begin{equation}\label{e7-1}
	d(T( f),f)\leq 1\leq\infty
	\end{equation}
	Now, it follows from Theorem  \ref{t0} (ii) that there exists a function $I:E\to E$
	with $I(0)=0$, which is a fixed point of $T$, such that $T^n (f) \to I$, i.e.,
	\begin{equation}\label{e8}
	I(x)=\lim_{n\to\infty} q_i^{-n}f(q_i^{n} x)
	\end{equation}
	for all $x\in E_1$.
	Since the integer $n_0$ of Theorem  \ref{t0} is $0$ then $f\in X^{\ast}$ which
	$$X^{\ast}=\{ y\in X:\ \ d(T^{n_0} (f),y)<\infty\}.$$
	By Theorem  \ref{t0} (iv) and (\ref{e7}) we obtain
	$$d(f,I)\leq \frac 1 {1-L} d(T( f), f)\leq \frac L {1-L}$$
	i.e., the inequality (\ref{e5}) is true for all $x \in E$.
	
By inequality (\ref{e1}), we have 
\begin{equation}\label{e9}
\lim_n q_i^{-n}\phi(q_i^{n}x,q_i^{n}y)=0
\end{equation}
	we replace $x$ with $2q_i^{n}x$ and $y$ with $2q_i^{n}y$ in (\ref{e2})consequently,
$$
\|D_{\lambda}f(2q_i^{n}x,2q_i^{n}y)\| \leq \phi(2q_i^{n}x,2q_i^{n}y)
$$
we replace  $\lambda=1$ in the last inequality and then	product	 it by $2q_i^{-n}$,	consequently
	 			\begin{align*}
	\Bigl\|\frac{1}{2q_i^{n}}f(q_i^{n}({x+y}))-\frac{1}{2q_i^{n}}f(2q_i^{n}x)-\frac{1}{2q_i^{n}}f(2q_i^{n}y)\Bigr\| \leq \frac{1}{2q_i^{n}}\phi(2q_i^{n}x,2q_i^{n}y)
	\end{align*}
		passing to the limit yields, 
	 			 $$\lim_{n\to\infty} 	\Bigl\|\frac{1}{2q_i^{n}}f(q_i^{n}({x+y}))-\frac{1}{2q_i^{n}}f(2q_i^{n}x)-\frac{1}{2q_i^{n}}f(2q_i^{n}y)\Bigr\|       =0	$$	
		 		 
		  then by 		 (\ref{e9})
		 $$\|I({x+y})-I(x)-I(y)\|=0$$
		 		 		we get for all $x,y\in X$ the equality
		 	$$I({x+y})=I(x)+I(y)$$
		 		there for,I is cauchy additive.
		 		If we replace $x$ with $y$ in (\ref{e2}) and $\lambda$ with $1$ , then (\ref{e2}) establishes
		 			$$\| 2\overline{\lambda}f(x)-2f(\lambda x)\|\leq \phi(x,x)\Rightarrow \| \overline{\lambda}f(x)-f(\lambda  x)\|\leq \frac1 2\phi(x,x)$$
	 	we replace  $x$ with $q_i^{n}x$ 
		 		 in the last inequality and product it by $q_i^{-n}$, we have
\begin{align*}
q_i^{-n}\| \overline{\lambda}f(q_i^{n}x)-f(q_i^{n}\lambda x)\| \leq  \frac1 2q_i^{-n}\phi(q_i^{n}x,q_i^{n}x)
\end{align*}
		 			consequently,passing to the limit yields, 
	 			 $$\lim_{n\to\infty} q_i^{-n}\| \overline{\lambda}f(q_i^{n}x)-f(q_i^{n}\lambda x)\| =0$$	
		 		 		 		 then by (\ref{e9})
		 		 $$\| \overline{\lambda}I(x)-I(\lambda  x)\|=0$$
		 				 		we get for all $x\in X$ and $\lambda \in \mathbb{T}_{\frac{1}{n_0}}^1$ the equality
		 		$$I( \lambda x)=\overline{\lambda}I(x)$$

		 		we show the last equality for all $\lambda \in \mathbb{C}$.
		 		It is enough to show that the last inequality is true for all $t\in (0,\infty)$, because if $\lambda\in\mathbb{C}$ then there exists $\theta\in [0,2\pi]$ such that $\lambda =|\lambda|e^{i\theta}$ and we have 
		 		$$I(\lambda x)=I(|\lambda| e^{i\theta}x)=|\lambda|I(e^{i\theta}x)	=|\lambda| e^{-i\theta}I(x)=\overline{\lambda}I(x),$$
		 		for all $x\in E$.

		 		 Suppose $\lambda\in\mathbb{T}^1$ ,then
there exists $\theta\in [0,2\pi]$ such that $\lambda =e^{i\theta}$. For $\lambda_1=e^{\frac{i\theta}{n_0}}$, we have $\lambda_1\in\mathbb{T}_{\frac{1}{n_0}}^1$ and
$$I( \lambda x)=I(\lambda_1^{n_0}x)=\overline{\lambda_1}^{n_0}I(x)=\overline{\lambda}I(x),$$
for all $x\in E$. If $\lambda\in n\mathbb{T}^1=\{nz\in \mathbb{C}: \|z\|=1\}$, for some $n\in \mathbb{ N}$, then by additivity of $I$ we have
$$\forall x\in E,\qquad  I( \lambda x)=\overline{\lambda}I(x).$$
If $t\in(0,\infty)$ then by  archimedean property of $\mathbb{C}$, there exists  $n\in (0,\infty)$ such that the point $(t,0)$ lies in the interior of circle with center at origin  and radius $n$. For $\alpha_1:=\alpha+\sqrt{n^2-\alpha^2}$ and $\alpha_2:=\alpha-\sqrt{n^2-\alpha^2}$, we have $\alpha=\frac{\alpha_1+\alpha_2}{2}$ and $\alpha_1,\alpha_2\in n\mathbb{T}^1$. So 
$$I(\alpha x)=I(\frac{\alpha_1+\alpha_2}{2} x)=\frac{\overline{\alpha_1}}{2}I(x)+\frac{\overline{\alpha_2}}{2}I(x)=\overline{\alpha}I(x)=\alpha I(x),$$
for all $x\in E$.
		 		Therefore, $I:E\to E$ is conjugate   $\mathbb{C}$-linear.
		 		we replace   $x$ with $q_i^{n}x$ and $y$ with $q_i^{n}y$ in (\ref{e3}) then product	 it by $q_i^{-2n}$	we have
	 					 		\begin{align*}
q_i^{-2n}\Bigl\|f(q_i^{2n}xy)-f(q_i^{n}y)f(q_i^{n}x)\|  \leq q_i^{-2n}\phi(q_i^{n}x,q_i^{n}y)
		 		\end{align*}
	 			 		consequently,passing to the limit yields, 
	 			 $$\lim_{n\to\infty}q_i^{-2n}\Bigl\|f(q_i^{2n}xy)-f(q_i^{n}y)f(q_i^{n}x)\|=0	$$	 
	 					 		then by (\ref{e9})
		 		$$\|I(xy)-I(y)I(x)\| =0$$
		 			 		we get for all $x,y\in X$ the equality
		 		$$ I(xy)=I(y)I(x)$$
		 		and from (\ref{e4}) we have
		 		$$I(I(x))=\lim_mq_i^{-m}f(q_i^m\lim_n q_i^{-n}f(q_i^{n}x))=x$$
		 		
		 		Further,let us show the unicity of $I$.In fact, assume the existence of another such involution $I'$ satisfies
(\ref{e5}) hence $I'(q_i^{n}x)=q_i^{n}I'(x)$ and we replace $x$ with $q_i^{n}x$  in (\ref{e5}) we have
 \begin{align*}
\|q_i^{-n}I'(q_i^{n}x)-q_i^{-n}f(q_i^{n}x)\| =q_i^{-n}\|I'(q_i^{n}x)-f(q_i^{n}x)\| \leq q_i^{-n}\frac{L^{1-i}} {1- L} \phi (q_i^{n}x,0)
\end{align*}
		 			consequently,passing to the limit yields, 
	 			 $$\lim_{n\to\infty} \|q_i^{-n}I'(q_i^{n}x)-q_i^{-n}f(q_i^{n}x)\| =q_i^{-n}\|I'(q_i^{n}x)-f(q_i^{n}x)\|       =0	$$	 
	 				 then by (\ref{e9})
		 		$\|I(x)-I'(x)\| =0$
		 			 		Hence $I=I'$ for all $x\in E$ .
		 		If I satisfies (\ref{e6})we replace  $x$ with $q_i^{n}x$  in (\ref{e6}) and	product	 it by $q_i^{-2n}$ we have 
\begin{align*}
	 q_i^{-2n}\Bigl|  \|q_i^{n}x_1f(q_i^{n}x)\|-\|q_i^{n}x\|^2\Bigr| \leq q_i^{-2n}\phi (q_i^{n}x,q_i^{n}x)
\end{align*}
		 	 	consequently,passing to the limit yields, 
	 			 $$\lim_{n\to\infty} q_i^{-2n}\Bigl|  \|q_i^{n}x_1f(q_i^{n}x)\|-\|q_i^{n}x\|^2\Bigr|  =0	$$	 
	 				 then by (\ref{e9})
		 	 	 $$\Bigl|\|xI(x)\|-\|x\|^2 \Bigr|=0$$
		 	 	 	
		 	 	 	It is enough to show that the last inequality is true for all $t\in (0,\infty)$, because if $\lambda\in\mathbb{C}$ then there exists $\theta\in [0,2\pi]$ such that $\lambda =|\lambda|e^{i\theta}$ and we have 
		 	 	 $$I( \lambda x)=I(|\lambda| e^{i\theta}x)=|\lambda|I(e^{i\theta}x)	=|\lambda| e^{-i\theta}I(x)=\overline{\lambda}I(x),$$
		 	 	 for all $x\in E$.	 	 		 	  		 	 
		 	 Hence $\|xI(x)\|=\|x\|^2$ for all $x\in E$. then $E$ is a $C^*$-algebra with involution $x^*=I(x)$, for all $x\in E$.
\end{proof}		 	
\begin{corollary}
	Let $E$  be  Banach algebra, $r\in (0,1)$, $\theta\in [0,\infty) $ and $f:E\to E$ with $f(0)=0$, such that
	\begin{align*}
&\|D_{\lambda}f(x,y) \|\leq \theta(\|x\|^r +\| y\|^r),\\[1mm]&
	\left\|f(xy)-f(y)f(x)\right\|\leq \theta (\|x\|^r +\| y\|^r),\\[1mm]&
	\lim_m 2^{-m}f(2^m\lim_n 2^{-n}f(2^{n}x))=x.
	\end{align*}
	for all $x,y \in E$ and $\lambda\in \mathbb{T}_{\frac{1}{n_0}}^1$, then there is a unique involution mapping $I:E\to E$ 	which satisfies 
	\begin{equation*}
	\|I(x)-f(x)\|\leq\frac {2\theta} {2-2^r} \|x\|^r
	\end{equation*}
	moreover if 
	\begin{equation*}
	\Bigl|  \|xf(x)\|-\|x\|^2 \Bigr|\leq{2\theta}  \|x\|^r
	\end{equation*}
	for all  $x \in E$, then $E$ is a $C^*$-algebra with involution $x^*=I(x)$, for all $x\in E$. 
	
\end{corollary}
\begin{proof}
	We put $$\phi(x,y):=\theta (\|x\|^r +\| y\|^r),$$ for all $x,y\in E$ and $L=2^{r-1}$ in theorem (\ref{t1}), then as a result, the sentence is obtained.
\end{proof}
\begin{corollary}
	Let $E$  be  Banach algebra, $r>1$, $\theta\in [0,\infty) $ and $f:E\to E$ with $f(0)=0$, such that
	\begin{align*}
	&\|D_{\lambda}f(x,y) \|\leq \theta(\|x\|^r +\| y\|^r),\\[1mm]&
	\left\|f(xy)-f(y)f(x)\right\|\leq \theta (\|x\|^r +\| y\|^r),\\[1mm]&
	\lim_m 2^{m}f(2^{-m}\lim_n 2^{n}f(2^{-n}x))=x.
	\end{align*}
	for all $x,y \in E$ and $\lambda\in \mathbb{T}_{\frac{1}{n_0}}^1$, then there is a unique involution mapping $I:E\to E$ 	which satisfies 
	\begin{equation*}
	\|I(x)-f(x)\|\leq\frac {2^{r}\theta} {2-2^r} \|x\|^r
	\end{equation*}
	moreover if 
	\begin{equation*}
	\Bigl|  \|xf(x)\|-\|x\|^2 \Bigr|\leq{2\theta}  \|x\|^r
	\end{equation*}
	for all  $x \in E$, then $E$ is a $C^*$-algebra with involution $x^*=I(x)$, for all $x\in E$. 
	
\end{corollary}
\begin{proof}
	We put $$\phi(x,y):=\theta (\|x\|^r +\| y\|^r),$$ for all $x,y\in E$ and $L=2^{1-r}$ in theorem (\ref{t1}), then as a result, the sentence is obtained.
\end{proof}
\begin{corollary}
		Let $E$  be  Banach algebra, $r\in (0,\infty)$, $\theta\in [0,\infty) $ and $f:E\to E$ with $f(0)=0$, such that
	\begin{align*}
	&\|D_{\lambda}f(x,y) \| \leq \theta \|x y\|^r,\\[1mm]&
	\left\|f(x,y)-f(y)f(x) \right\| \leq \theta \|x y\|^r,\\[1mm]&
	\lim_m q^{-m}f(q^m\lim_nq^{-n}f(q^{n}x))=x.
	\end{align*}
	for all $x,y \in E$,  $\lambda \in \mathbb{T}_{\frac{1}{n_0}}^1$ and  $q=2$ and $q=\frac 1 2$ respectively when $r<\frac 1 2$ and $r>\frac 1 2$, then $f$ is unique involution on $E$,
	moreover if 
	\begin{equation*}
	\Bigl|  \|xf(x)\|-\|x\|^2 \Bigr|\leq{\theta}  \|x\|^{2r}
	\end{equation*}
	for all  $x \in E$, then $E$ is a $C^*$-algebra with involution $x^*=f(x)$, for all $x\in E$. 
	
\end{corollary}
\begin{proof}
 We put $$\phi(x,y):=\theta\|xy\|^r,$$ for all $x,y\in E$.
 We and $L=q^{2r-1}$ in theorem (\ref{t1}), then as a result, the sentence is obtained.
\end{proof}
	\bibliography{nafis10-1}
	\bibliographystyle{mmn}
\end{document}